\documentclass{amsart}

\usepackage{bm}
\usepackage{amsthm}
\usepackage[usenames,dvipsnames]{pstricks}
\usepackage{epsfig}
\usepackage{pst-grad} 
\usepackage{pst-plot} 
\usepackage[small,nohug]{diagrams}

\newtheorem{theorem}{Theorem}[section]
\newtheorem{corollary}[theorem]{Corollary}
\newtheorem{lemma}[theorem]{Lemma}
\newtheorem{proposition}[theorem]{Proposition}

\theoremstyle{remark}

\theoremstyle{definition}
\newtheorem{definition}[theorem]{Definition}

\newcommand{\dist}{\mathrm{d}}
\newcommand{\ind}{\text{ind-}}

\begin{document}

\title{A coarse invariant for all metric spaces}

\author{M.~DeLyser}
\address{Saint Francis University, Loretto, PA 15940}
\email{mrd100@@francis.edu}

\author{B.~LaBuz}
\address{Saint Francis University, Loretto, PA 15940}
\email{blabuz@@francis.edu}

\author{B.~Wetsell}
\address{Saint Francis University, Loretto, PA 15940}
\email{brwst6@@francis.edu}

\date{\today}

\begin{abstract}
In \cite{borno} an invariant of metric spaces under bornologous equivalences is defined. In \cite{coarse} this invariant is extended to coarse equivalences. In both papers the invariant is defined for a class of metric spaces called sigma stable. This paper extends the invariant to all metric  spaces and also gives an example of a space that is not sigma stable.
\end{abstract}

\maketitle
\tableofcontents

\medskip 
\medskip

\section{Introduction}

Large scale geometry is the study of the large scale structure of metric spaces. Continuity is a small scale property of a function; one only needs to check the property for small distances. A property dual to continuity (in fact uniform continuity) is bornology. A function $f:X\to Y$ is bornologous if for each $N>0$ there is an $M>0$ such that for every $x,y\in X$, if $\dist(x,y)\leq N$, $\dist(x,y)\leq M$ \cite{roe}. Notice the only change from the definition of uniform continuity is the swapping of the orders of the real numbers $N$ and $M$ in the latter part of the statement. Bornology is a large scale property of a function; one only needs to check the property for large distances.

Roe \cite{roe} defines the coarse category with metric spaces as objects and close equivalence classes of coarse functions as morphisms. We say two functions $f:X\to Y$ and $g:X\to Y$ are close if there is some constant $K$ with $\dist(f(x),g(x))\leq K$ for all $x\in X$. A function is metrically proper if the inverse image of bounded sets are bounded. A function is coarse if it is bornologous and proper. Two metric spaces $X$ and $Y$ are coarsely equivalent if there are coarse functions $f:X\to Y$ and $g:Y\to X$ such that $g\circ f$ is close to the identity function on $X$ and $f\circ g$ is close to the identity function on $Y$.

It is of interest to study the isomorphisms of the coarse category: when are two metric spaces coarsely equivalent? Typically to show two spaces are coarsely equivalent we construct the coarse functions $f$ and $g$ that form the coarse equivalence. The basic example of coarsely equivalent metric spaces is $\mathbb R$ and $\mathbb Z$ under the usual metrics. We take $f:\mathbb Z\to\mathbb R$ to be the inclusion function and $g:\mathbb R\to\mathbb Z$ to be the floor function. It is an easy exercise to check that these functions form a coarse equivalence.

How do we show that two spaces are not coarsely equivalent? We cannot check that all possible functions do not form an equivalence. In \cite{borno} an invariant of metric spaces under bornologous equivalences is defined. This invariant provides a way to detect if two spaces are not bornologously equivalent. The bornologous category is a more restrictive category than the coarse category; the compositions are required to be the identity on the nose. Equivalently, a bornologous bijection $f:X\to Y$ whose inverse is bornologous is required. Thus if two spaces are bornologously equivalent then they are coarsely equivalent. In \cite{coarse} the invariant is extended to the coarse category. In both of these papers the invariant is only defined for a class of spaces called $\sigma$-stable spaces. In this paper we simultaneously extend the results from \cite{borno} to the coarse category and to all metric spaces.

We review the construction from \cite{borno}. Suppose $N>0$. Given a metric space $X$ and a basepoint $x_0\in X$, an $N$-sequence in $X$ based at $x_0$ is an infinite list $x_0,x_1,\ldots$ of points in $X$ such that $\dist(x_i,x_{i+1})\leq N$ for all $i\geq 0$. The following is a nice interpretation of a bornologous function. A function $f:X\to Y$ is bornologous if and only if for each $N>0$ there is an $M>0$ such that $f$ sends $N$-sequences in $X$ to $M$-sequences in $Y$.

We are only interested in sequences that go to infinity. An $N$-sequence $x_0,x_1,x_2,\ldots$ goes to infinity, $(x_i)\to\infty$, if $\lim_{i\to \infty}\dist(x_i,x_0)=\infty$. We want to consider an equivalence relation between sequences. Given two $N$-sequences $s$ and $t$ in $X$ based at $x_0$ that go to infinity define $s$ and $t$ to be related, $s\sim t$, if $s$ is a subsequence of $t$ or $t$ is a subsequence of $s$. If $t$ is a subsequence of $s$ we say that $s$ is a supersequence of $t$. Define $s$ and $t$ to be equivalent, $s\approx t$,  if there is a finite list of sequences $s_i$ such that $s\sim s_1\sim s_2\sim \cdots \sim s_n\sim t$. Let $[s]_N$ denote the equivalence class of $s$ and let $\sigma_N(X,x_0)$ be the set of equivalence classes.

For each integer $N>0$ there is a function $\phi_N:\sigma_N(X,x_0)\to \sigma_{N+1}(X,x_0)$ that sends an equivalence class $[s]_N$ to the equivalence class $[s]_{N+1}$. A space $X$ is called $\sigma$-stable if there is an integer $K>0$ such that $\phi_N$ is a bijection for each $N\geq K$. If $X$ is $\sigma$-stable define $\sigma(X,x_0)$ to be the cardinality of $\sigma_K(X,x_0)$. The following theorem of \cite{borno} says that it is an invariant.

\begin{theorem}
Suppose $f:X\to Y$ is a bornologous equivalence between metric spaces. Let $x_0$ be a basepoint of $X$ and set $y_0=f(x_0)$. Suppose $X$ and $Y$ are $\sigma$-stable. Then $\sigma(X,x_0)=\sigma(Y,y_0)$.
\end{theorem}

\section{The invariant}

We wish to extend the above invariant to all metric spaces. We do so by considering the direct sequence $\{\sigma_N(X,x_0),\phi_N\}$. A direct sequence of sets is a family of sets $\{X_i\}$, $i\in\mathbb N$, together with a family of functions $\{\phi_i:X_i\to X_{i+1}\}$ called bonding functions \cite{Bourbaki}. For $i<j$ we write $\phi_{j-1}\circ\cdots\circ\phi_{i+1}\circ\phi_i=\phi_{ij}$ so that $\phi_{ij}:X_i\to X_j$. We also let $\phi_{ii}$ be the identity on $X_i$.

Typically a morphism between direct sequences are defined as level morphisms. We find it more convenient to allow more general morphisms. We define a morphism from a direct sequence $\{X_i,\phi_i\}$ to a direct sequence $\{Y_i,\psi_i\}$ as a sequence of functions $f_i:X_i\to Y_{u(i)}$ where $u:\mathbb N\to \mathbb N$ such that if $i<j$, $\psi_{u(i)u(j)}\circ f_i=f_j\circ \phi_{ij}$. We can assume that $u(i)\geq i$. If not, we create the new sequence of functions $\{f'_i=\psi_{u(i)i}\circ f_i\}$. Similarly we can assume that if $i<j$, $u(i)<u(j)$. If not we define $f'_j=\psi_{u(j)u(i)}\circ f_j$.

We define two direct sequences $\{X_i,\phi_i\}$ and $\{Y_i,\psi_i\}$ to be equivalent if there are morphisms $\{f_i:X_i\to Y_{u(i)}\}$ and $\{g_i:Y_i\to X_{v(i)}\}$ such that $g_{u(i)}\circ f_i=\phi_{i~v(u(i))}$ and $f_{v(i)}\circ g_i=\psi_{i~u(v(i))}$.

If two direct sequences are equivalent then there is a bijection between the corresponding direct limits (see the Appendix).

\begin{definition}
Let $X$ be a metric space with basepoint $x_0$. Consider the direct sequence $\{\sigma_N(X,x_0),\phi_N\}$ where $\phi_N$ sends an equivalence class $[s]_N$ to $[s]_{N+1}$. We denote this sequence as $\ind\sigma(X,x_0)$ and its direct limit $\varinjlim \sigma_N(X,x_0)$ as $\sigma(X,x_0)$. The ind stands for inductive sequence, another term for direct sequence.
\end{definition}

Notice that in the case of a $\sigma$-stable space $X$, we have that the cardinality of $\sigma(X,x_0)$ is equal to the value $\sigma(X,x_0)$ defined in \cite{borno} and \cite{coarse} so this definition can be thought of as a generalization of that concept that applies to all metric spaces.

First we show that the choice of basepoint does not matter. Thus we can suppress the notation for basepoint and just write $\ind\sigma(X)$.

\begin{proposition}
Let $X$ be a metric space with basepoint $x_0$. Given $y_0\in X$, $\ind\sigma(X,x_0)$ is equivalent to $\ind\sigma(X,y_0)$.
\end{proposition}

\begin{proof}
Choose an integer $M\geq\dist(x_0,y_0)$. For each $N<M$, define $f_N:\sigma_N(X,x_0)\to\sigma_M(X,y_0)$ to send $[s]_N\in\sigma_N(X,x_0)$, $s=x_0,x_1,\ldots$, to the equivalence class of the sequence $y_0,x_0,x_1,x_2,\ldots$. For $N\geq M$, we define $f_N:\sigma_N(X,x_0)\to \sigma_N(X,y_0)$ in a similar fashion, attaching the point $y_0$ to the beginning of a sequence. We define functions $g_N:\sigma_N(X,y_0)\to \sigma_M(X,x_0)$ for $N<M$ and $g_N:\sigma_N(X,y_0)\to \sigma_N(X,x_0)$ for $N\geq M$ analogously.

Let us see that the composition $g_M\circ f_N$ is equal to $\phi_{NM}$ for $N<M$. The composition sends the equivalence class of a sequence $x_0,x_1,x_2,\ldots$ to the equivalence class of a sequence $x_0,y_0,x_0,x_1,x_2,\ldots$ which is clearly the same as the equivalence class of $x_0,x_1,\ldots$. Similarly we have that the composition $g_N\circ f_N$ is equal to the identity on $\sigma_N(X,x_0)$ for $N\geq M$. The opposite compositions are similar as well.
\end{proof}

\begin{theorem}
Suppose $X$ and $Y$ are coarsely equivalent metric spaces. Then $\ind\sigma(X)$ is equivalent to $\ind\sigma(Y)$.
\end{theorem}

\begin{proof}
Suppose $f:X\to Y$ and $g:Y\to X$ make up the coarse equivalence. Given $N\in\mathbb N$, since $f$ is bornologous there is an $M>0$ so that if $\dist(x,y)\leq N$, $\dist(f(x),f(y))\leq M$. We can assume $M$ is an integer greater than $N$. Set $f(x_0)=y_0$. Thus we have a well defined function $f_N:\sigma_N(X,x_0)\to \sigma_M(Y,y_0)$ that sends the equivalence class of a sequence $x_0,x_1,\ldots$ to $f(x_0),f(x_1),\ldots$. 

For the opposite morphism, let $K$ be an integer so that $\dist(g(f(x)),x)\leq K$ for all $x\in X$. Given $N\in\mathbb N$, since $g$ is bornologous there is an $M>0$ so that if $\dist(x,y)\leq N$, $\dist(g(x),g(y))\leq M$. We can take $M$ to be an integer greater than $N$ and $K$. We define a function $f_N:\sigma_N(Y,y_0)\to \sigma_M(X,x_0)$ that sends the equivalence class of a sequence $y_0,y_1,\ldots$ to $x_0,g(y_0),g(y_1),\ldots$.

We check that the following diagram commutes.

\begin{diagram}
\sigma_L(X,x_0)  &             &                 \\
                 & \luTo^{g_M} &                  \\
\uTo^{\phi_{NL}} &             & \sigma_M(Y,y_0) \\
                 & \ruTo_{f_N} &                 \\
\sigma_N(X,x_0)  &             &                \\
\end{diagram}

Let $[s]_N\in\sigma_N(X,x_0)$, say $s=x_0,x_1,\ldots$. Then $g_M(f_N([s]_N))$ is the equivalence class of the sequence $x_0,g(f(x_0)),g(f(x_1)),\ldots$. This sequence is equivalent to $x_0,x_1,\ldots$ in $\sigma_L(X,x_0)$ since the sequence $x_0,g(f(x_0)),x_0,x_1,g(f(x_1)),x_1,\ldots$ is a supersequence of both.

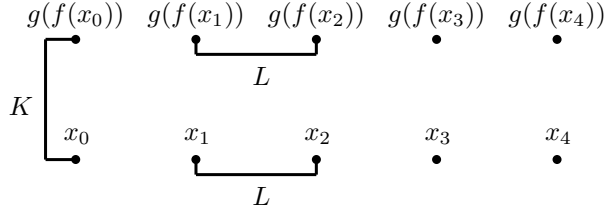
\begin{figure}[h]

\begin{pspicture}(0,-1.3829688)(8.242812,1.3829688)
\psdots[dotsize=0.12](0.9809375,0.88453126)
\psdots[dotsize=0.12](2.5809374,0.88453126)
\psdots[dotsize=0.12](4.1809373,0.88453126)
\psdots[dotsize=0.12](5.7809377,0.88453126)
\psdots[dotsize=0.12](7.3809376,0.88453126)
\psdots[dotsize=0.12](0.9809375,-0.71546876)
\psdots[dotsize=0.12](2.5809374,-0.71546876)
\psdots[dotsize=0.12](4.1809373,-0.71546876)
\psdots[dotsize=0.12](5.7809377,-0.71546876)
\psdots[dotsize=0.12](7.3809376,-0.71546876)
\usefont{T1}{ptm}{m}{n}
\rput(0.25234374,-0.00546875){$K$}
\usefont{T1}{ptm}{m}{n}
\rput(3.4523437,0.39453125){$L$}
\usefont{T1}{ptm}{m}{n}
\rput(1.0023438,-0.40546876){$x_0$}
\usefont{T1}{ptm}{m}{n}
\rput(2.6023438,-0.40546876){$x_1$}
\usefont{T1}{ptm}{m}{n}
\rput(5.802344,-0.40546876){$x_3$}
\usefont{T1}{ptm}{m}{n}
\rput(7.4023438,-0.40546876){$x_4$}
\usefont{T1}{ptm}{m}{n}
\rput(0.98234373,1.1945312){$g(f(x_0))$}
\usefont{T1}{ptm}{m}{n}
\rput(2.5823438,1.1945312){$g(f(x_1))$}
\usefont{T1}{ptm}{m}{n}
\rput(4.182344,1.1945312){$g(f(x_2))$}
\usefont{T1}{ptm}{m}{n}
\rput(5.782344,1.1945312){$g(f(x_3))$}
\usefont{T1}{ptm}{m}{n}
\rput(7.382344,1.1945312){$g(f(x_4))$}
\usefont{T1}{ptm}{m}{n}
\rput(3.4523437,-1.2054688){$L$}
\psline[linewidth=0.04cm](0.9809375,0.88453126)(0.5809375,0.88453126)
\psline[linewidth=0.04cm](0.5809375,0.88453126)(0.5809375,-0.71546876)
\psline[linewidth=0.04cm](0.5809375,-0.71546876)(0.9809375,-0.71546876)
\usefont{T1}{ptm}{m}{n}
\rput(4.202344,-0.40546876){$x_2$}
\psline[linewidth=0.04cm](2.5809374,-0.71546876)(2.5809374,-0.91546875)
\psline[linewidth=0.04cm](2.5809374,-0.91546875)(4.1809373,-0.91546875)
\psline[linewidth=0.04cm](4.1809373,-0.91546875)(4.1809373,-0.71546876)
\psline[linewidth=0.04cm](2.5809374,0.88453126)(2.5809374,0.6845313)
\psline[linewidth=0.04cm](2.5809374,0.6845313)(4.1809373,0.6845313)
\psline[linewidth=0.04cm](4.1809373,0.6845313)(4.1809373,0.88453126)
\end{pspicture} 

\caption{The two sequences are equivalent.}
\end{figure}

The proof that the diagram below commutes is similar.

\begin{diagram}
                 &             &  \sigma_L(Y,y_0)   \\
                & \ruTo^{f_M} &                  \\
\sigma_M(X,x_0)  &             & \uTo_{\psi_{NL}} \\
                 & \luTo_{g_N} &                 \\
                 &             & \sigma_N(Y,y_0)   \\
\end{diagram}

\end{proof}

\section{A space that is not $\sigma$-stable}

As motivation for the generalization of the construction of \cite{borno} to all metric spaces we give an example of a space that is not $\sigma$-stable. This example will also serve as an illustration of the invariant and as motivation for the use of direct sequences rather than merely a direct limit.

Given a family of pointed metric spaces $(X_{\alpha},x_{\alpha})$ we define their metric wedge $\bigvee (X_{\alpha},x_{\alpha})$ as the wedge with the following metric. Given $x,y\in \bigvee (X_{\alpha},x_{\alpha})$, $\dist(x,y)=\begin{cases} 
\dist (x,y) & \text{if } x,y\in X_{\alpha} \text{ for some } \alpha\\
\dist (x,x_{\alpha})+\dist(x_{\beta},y) & \text{if } x\in X_\alpha \text{ and } y\in X_{\beta} \text{ with } \alpha\neq\beta
\end{cases}$.

Let the open book $B$ be the metric wedge of rays $B_i= [0,\infty)$, $i\in\mathbb N$, based at the points $0$. Denote the wedge point as $x_0$.

\begin{lemma}\label{BookLemma}
Suppose $s$ is an $N$-sequence in $B$ based at the wedge point $x_0$ that goes to infinity. Then there is an $M>0$ and $k\in\mathbb N$ so that for each $n\geq M$, $s_n$ lies on the ray $B_k$. Further, if $t$ is an $N$-sequence in $B$ based at the $x_0$ that goes to infinity with $t\sim s$, then there is an $R>0$ such that for all $n\geq R$, $t_n$ lies on $B_k$.
\end{lemma}

\begin{proof}
Since $s$ goes to infinity there is a $M>0$ such that $\dist(s_n,x_0)\geq N+1$ for all $n\geq M$. Say $s_M\in B_k$. Since the distance from $s_M$ to any other ray is at least $N+1$, $s_{M+1}$ must also lie on $B_k$. By induction, we can see that $s_n$ lies on $B_k$ for all $n\geq M$.

Since $t$ goes to infinity, there is $P>0$ such that $\dist(t_n,x_0)\geq N+1$ for all $n\geq P$. Choose $R\geq P$ so that $t_R=s_n$ for some $n\geq M$. Thus $t_R$ is on $B_k$. We have $t_n$ lying on $B_k$ for all $n\geq R$ as above. 
\end{proof}

\begin{theorem}\label{indsigmaB}
Let $s_i$ be the sequence in $B$ based at $x_0$ and lying on $B_i$ where each term $s_{in}=n$, $n\in\mathbb N\cup\{0\}$. Let $N\geq 1$. Then $\sigma_N(B,x_0)=\{[s_1],[s_2],\ldots\}$.
\end{theorem}

\begin{proof}
First we show that if $i\neq j$, $[s_i]\neq[s_j]$. Suppose to the contrary that $[s_i]=[s_j]$. Then there is a list $t_1,\ldots,t_k$ of $N$-sequences going to infinity with $s_i\sim t_1\sim t_2\sim\cdots\sim t_k\sim s_j$. By Lemma \ref{BookLemma} $t_1$ must eventually lie on $B_i$. Likewise $t_2$, $t_3$, and finally $s_j$ must eventually lie on $B_i$. But $s_j$ lies entirely on $B_j$, a contradiction.

Now suppose $[t]\in \sigma_{N}(B,x_0)$. By Lemma \ref{BookLemma} there is an $M>0$ and $k\in\mathbb N$ so that for all $m\geq M$, $t_m$ lies on $B_k$. We show that $[t]=[s_k]$. We create a new sequence $r$ equivalent to $t$ whose terms all lie on $B_k$. We can then see that $r$ is equivalent to $s_k$ as in the proof of \cite[Theorem 3.6]{borno}. We know that $t_m$ lies on $B_k$ for $m\geq M$. Let $t_q$ be the first term of the sequence that is not the basepoint or a point on $B_k$. If no such point exists we are done. Let $t_p$ be the first point of the sequence after $t_q$ that is the basepoint or a point on $B_k$. Define $r_1$ to be the sequence $t_0,\ldots,t_{q-1},t_{p},t_{p+1},\ldots$. Now $t_{q-1}$ and $t_p$ both lie on $B_k$ and are distance at most $N$ from the basepoint. Thus $\dist(t_{q-1},t_p)\leq N$ so $r_1$ is an $N$-sequence and $t\sim r_1$. We continue by induction, finally defining a sequence $r$ whose terms all lie on $B_k$.
\end{proof}

We define a subspace of $B$ that we call the discrete open book $D$. Let $D_i=\{in:n\in\mathbb N\cup \{0\}\}$. Thus $D_i$ has points that are distance $i$ apart. Define $D=\bigvee D_i$ based at the points $0$. Again, denote the wedge point as $x_0$. The following theorem implies that $D$ is not $\sigma$-stable.

\begin{theorem}\label{indsigmaD}
Let $s_i$ be the sequence in $D$ based at $x_0$ and lying on $D_i$ where each term $s_{in}=in$, $n\in\mathbb N\cup\{0\}$. Let $N\geq 1$. Then $\sigma_{N}(D,x_0)=\{[s_1],[s_2],\ldots ,[s_N]\}$.
\end{theorem}

\begin{proof}
The proof is similar to that of Theorem \ref{indsigmaB}. Given $[t]\in \sigma_{N}(D,x_0)$, by Lemma \ref{BookLemma} there is an $M>0$ and $k\in\mathbb N$ so that for all $m\geq M$, $t_m$ lies on $B_k$. The difference here is that we must have $k\leq N$ since the distance between successive points on $D_k$ is $k$ and $t$ is an $N$-sequence that goes to infinity.
\end{proof}

\begin{corollary}
The open book $B$ and the discrete open book $D$ are not coarsely equivalent.
\end{corollary}

\begin{proof}
According to Theorem \ref{indsigmaD} $\ind\sigma(D)$ is the sets $\{[s_1],[s_2],\ldots,[s_N]\}$ with the bonding function $\phi_N:\{[s_1],[s_2],\ldots,[s_N]\}\to\{[s_1],[s_2],\ldots,[s_{N+1}]\}$ being inclusion. According to \ref{indsigmaB}, $\ind\sigma(B)$ is the sets $\{[s_1],[s_2],\ldots\}$ with the bonding function $\psi_N:\{[s_1],[s_2],\ldots\}\to\{[s_1],[s_2],\ldots\}$ being the identity. These direct sequences are not equivalent since the identity function $\sigma_N(B)\to\sigma_L(B)$ cannot factor as $\sigma_N(B)\to\sigma_M(D)\to\sigma_L(B)$ since $\sigma_M(D)$ is finite.
\end{proof}

The previous corollary illustrates the power of studying the direct sequence rather than merely the direct limit. We have that $\sigma(D)\cong \sigma(B)\cong \mathbb N$.

\appendix
\section{Direct limits}

Given a direct sequence $\{X_i,\phi_i\}$, its direct limit $\varinjlim X_i$ is defined as follows. Consider the disjoint union $\bigsqcup X_i$. We define an equivalence relation on $\bigsqcup X_i$ as follows. Given $x_i\in X_i$ and $x_j\in X_j$, $x_i$ is related to $x_j$ if there is some $k\geq i,j$ such that $\phi_{ik}(x_i)=\phi_{jk}(x_j)$. The set of equivalence classes is the direct limit.

A morphism $\{f_i:X_i\to Y_{u(i)}\}$ between direct sequences $\{X_i,\phi_i\}$ and $\{Y_i,\psi_i\}$ induces a function $f:\varinjlim X_i\to\varinjlim Y_i$. Given $[x_i]\in\varinjlim X_i$, $x_i\in X_i$, set $f([x_i])=[f_i(x_i)]$. This function is well defined since if $\phi_{ik}(x_i)=\phi_{jk}(x_j)$, $f_k(\phi_{ik}(x_i))=f_k(\phi_{jk}(x_j))$ so $\psi_{u(i)u(k)}(f_i(x_i))=\psi_{u(j)u(k)}(f_i(x_i))$.

\begin{proposition}
Suppose two direct sequences $\{X_i,\phi_i\}$ and $\{Y_i,\psi_i\}$ are equivalent. Then $\varinjlim X_i$ and $\varinjlim Y_i$ are equivalent as sets.
\end{proposition}

\begin{proof}
Let $\{f_i:X_i\to Y_{u(i)}\}$ and $\{g_i:Y_i\to X_{v(i)}\}$ be morphisms that make up the equivalence. We show that the compositions of the induced functions $f$ and $g$ are the identities. First consider $g\circ f$. Suppose $[x_i]\in\varinjlim X_i$, $x_i\in X_i$. Then $g(f([x_i]))=g([f_i(x_i)])=[g_{u(i)}(f_i(x_i))]$. But $g_{u(i)}(f_i(x_i))=\phi_{i~v(u(i))}(x_i)$ so $[g_{u(i)}(f_i(x_i))]=[x_i]$. A similar argument shows that $f\circ g$ is the identity on $Y$.
\end{proof}

\end{document}